\documentclass[english,a4paper]{amsart}
\usepackage[T1]{fontenc}
\usepackage[latin9]{inputenc}
\usepackage{amsthm}
\usepackage{graphicx}
\usepackage{setspace}
\usepackage{amssymb}
\usepackage[all]{xy}
\numberwithin{equation}{section} 
\numberwithin{figure}{section} 
\theoremstyle{plain}
\newtheorem{thm}{Theorem}[section]
  \theoremstyle{plain}
  \newtheorem{fact}[thm]{Fact}
  \theoremstyle{plain}
  \newtheorem{prop}[thm]{Proposition}
  \theoremstyle{remark}
  \newtheorem{rem}[thm]{Remark}
  \theoremstyle{plain}
  \newtheorem{lem}[thm]{Lemma}
  \theoremstyle{remark}
  \newtheorem*{acknowledgement*}{Acknowledgement}

\makeatletter
 \usepackage{xypic}

\newcommand{\N}{\mathbb{N}}

\newcommand{\Q}{\mathbb{Q}}
\newcommand{\R}{\mathbb{R}}

\renewcommand{\phi}{\varphi}
\renewcommand{\Pi}{\pi}
\renewcommand{\hat}{\widehat}
\renewcommand{\lg}{\log}
\renewcommand{\Im}{\mathrm{Im}}

\newcommand{\diam}{\mathrm{diam}}
\newcommand{\HD}{\mathrm{HD}}

\DeclareMathOperator*{\sumsum}{\sum\cdots\sum}
\DeclareMathOperator*{\sums}{\sum\sum}

\makeatother

\usepackage{babel}

\begin{document}
\selectlanguage{english}

\title[The arithmetic-geometric scaling spectrum]{The arithmetic-geometric scaling spectrum for continued fractions}

\date{\today}

\author{Johannes Jaerisch and Marc Kesseböhmer}

\address{Universität Bremen, Bibliothekstrasse 1, 28356 Bremen}

\email{mhk@math.uni-bremen.de, jogy@math.uni-bremen.de}

\urladdr{www.math.uni-bremen.de/stochdyn}

 \begin{abstract}
To compare continued fraction digits with the denominators of the
corresponding approximants we introduce the a\-rith\-me\-tic-ge\-ometric
scaling. We will completely determine its multifractal spectrum by
means of a number theoretical free energy function and show that the
Hausdorff dimension of sets consisting of irrationals with the same
scaling exponent coincides with the Legendre transform of this free
energy function. Furthermore, we identify the asymptotic of the local
behaviour of the spectrum at the right boundary point and discuss
a connection to the set of irrationals with continued fraction digits
exceeding a given number which tends to infinity.
\end{abstract}


\maketitle

\section{Introduction and statement of results}

Many ergodic and dynamical properties of the Gauss map \begin{equation}
T\colon\mathbb{I}\to\mathbb{I},\: x\mapsto\frac{1}{x}-\left\lfloor \frac{1}{x}\right\rfloor ,\label{eq:GaussMap}\end{equation}
 where $\mathbb{I}:=\left(0,1\right)\setminus\Q$, have been studied
in great detail (e.g. \cite{Kuzmin1928,Wirsing:74}). Also its close
relation to the Riemann $\zeta$-function via its Mellin transform
is well-known, i.e. $\zeta\left(s\right)-1=\left(s-1\right)^{-1}-s\int_{0}^{1}T\left(x\right)x^{s-1}\, dx$.
From the ergodic theoretical point of view the Gauss map reflects
mainly geometric features of continued fraction expansion, whereas
the Riemann $\zeta$-function reflects important arithmetic properties.
In this paper we establish an approach to quantify the difference
of these two aspects. For this let us begin with some elementary observations.

Every $x\in\mathbb{I}$ has a unique representation by its regular
continued fraction expansion, i.e. we have a bijective map $\pi\colon\mathbb{I}\longrightarrow\N^{\N}$,
where $\pi\left(x\right)=\left(a_{i}\left(x\right)\right)_{i\in\N}$
with \[
x=\cfrac{1}{a_{1}\left(x\right)+\cfrac{1}{a_{2}\left(x\right)+\cfrac{1}{a_{3}\left(x\right)+\cdots}}}.\]
 For $n\in\N$ the $n$-th convergent of $x\in\mathbb{I}$ is given
by the reduced fraction \[
\frac{p_{n}\left(x\right)}{q_{n}\left(x\right)}:=\cfrac{1}{a_{1}\left(x\right)+\cfrac{1}{\ddots+\cfrac{1}{a_{n-1}\left(x\right)+\cfrac{1}{a_{n}\left(x\right)}}}},\]
which is uniquely determined by the first $n$ digits of its continued
fraction expansion. Hence, we will also use the notation $q_{n}\left(\omega\right):=q_{n}\left(x\right)$
and $p_{n}\left(\omega\right):=p_{n}\left(x\right)$ whenever $\omega$
is an infinite or finite word of length at least $n$ over the alphabet
$\N$ such that the vector of the first $n$ entries coincide with
$\left(a_{1}\left(x\right),\ldots,a_{n}\left(x\right)\right)$. For
the denominator we then have the following recursive formula \begin{eqnarray}
q_{n}\left(x\right) & = & a_{n}\left(x\right)q_{n-1}\left(x\right)+q_{n-2}\left(x\right),\label{eq:pnqn recursion}\end{eqnarray}
with $q_{-1}=0$ and $q_{0}=1$. From this one immediately verifies
that \begin{equation}
\prod_{i=1}^{n}a_{i}\left(x\right)\le q_{n}\left(x\right)\le2^{n}\prod_{i=1}^{n}a_{i}\left(x\right),\label{eq:qn versus prod ai einfach}\end{equation}
showing that the arithmetic expression $\prod_{i=1}^{n}a_{i}\left(x\right)$
does not differ too much from the geometric term $q_{n}\left(x\right)$.
Yet, the two terms may grow on different exponential scales. Our main
aim is to investigate the fluctuation of the asymptotic exponential
scaling. For this let us define the \emph{arithmetic-geometric scaling
of $x\in\mathbb{I}$} by $\lim_{n\rightarrow\infty}\lg\prod_{i=1}^{n}a_{i}\left(x\right)/\lg q_{n}\left(x\right)$
if the limit exists. The fluctuation of this quantity is captured
in the level sets \[
\mathcal{F}_{\alpha}:=\left\{ x\in\mathbb{I}:\lim_{n\rightarrow\infty}\frac{\lg\left(\prod_{i=1}^{n}a_{i}\left(x\right)\right)}{\lg q_{n}\left(x\right)}=\alpha\right\} \]
 for a prescribed scaling $\alpha\in\R$.

The following list of facts give a first impression of these level
sets. Their proofs will be postponed until Subsection \ref{sub:Proofs-of-facts}.
\begin{fact}
\label{fact1} For $\alpha\notin\left[0,1\right]$ we have $\mathcal{F}_{\alpha}=\emptyset$. 
\end{fact}
\begin{fact}
\label{fact2} The noble numbers (i.e. those numbers whose continued
fraction expansion eventually contain only\/ $1$'s) are contained
in $\mathcal{F}_{0}$. 
\end{fact}
\begin{fact}
\label{fact3}For $k\in\N$ the quadratic surd $\pi^{-1}\left(k,k,\dots\right)$
lies in $\mathcal{F}_{\alpha\left(k\right)}$, where \[
\alpha\left(k\right):=\frac{\lg k}{-\lg\left(-k/2+\sqrt{k^{2}/4+1}\right)}\in[0,1)\textrm{ and }\lim_{k\rightarrow\infty}\alpha\left(k\right)=1.\]

\end{fact}
\begin{fact}
\label{fact4}The numbers having a continued fraction expansion with
digits tending to infinity are contained in $\mathcal{F}_{1}$, i.e.
$\mathcal{G}:=\left\{ x\in\mathbb{I}:a_{i}\left(x\right)\rightarrow\infty\right\} \subset\mathcal{F}_{1}$. 
\end{fact}
\begin{fact}
\label{fact5}We have for $\lambda$-almost every $x\in\left(0,1\right)$
that \[
\lim_{n\longrightarrow\infty}\frac{\lg\prod_{i=1}^{n}a_{i}\left(x\right)}{\lg q_{n}\left(x\right)}=\frac{12\lg2}{\pi^{2}}\lg\left(K_{0}\right):=\alpha_{0}=0.8325\ldots,\]
 where $\lambda$ denotes the Lebesgue measure restricted to\/ $[0,1]$
and \[
K_{0}:=\prod_{k\in\N}\left(1+\left(k\left(k+2\right)\right)^{-1}\right)^{\log k/\log2}\]
 the Khintchin constant (cf. \cite{Khinchin:56}). Consequently, we
have $\lambda\left(\mathcal{F}_{\alpha_{0}}\right)=1$. 
\end{fact}
\begin{fact}
\label{fact6}Also for later use let us define \[
\mathcal{F}_{\alpha}^{*}:=\begin{cases}
\left\{ x\in\mathbb{I}:\limsup_{n\rightarrow\infty}\lg\left(\prod_{i=1}^{n}a_{i}\left(x\right)\right)/\lg q_{n}\left(x\right)\geq\alpha\right\} , & \alpha\geq\alpha_{0},\\
\left\{ x\in\mathbb{I}:\liminf_{n\rightarrow\infty}\lg\left(\prod_{i=1}^{n}a_{i}\left(x\right)\right)/\lg q_{n}\left(x\right)\leq\alpha\right\} , & \alpha<\alpha_{0}.\end{cases}\]
 Then for $\alpha_{q}:=1-\left(q^{2}\log\left(q\right)\right)^{-1}$,
$q>2$, we have \[
\mathcal{I}_{q}:=\left\{ x\in\mathbb{I}:a_{i}\left(x\right)\geq q,\, i\in\N\right\} \subset\mathcal{F}_{\alpha_{q}}^{*}.\]

\end{fact}
\begin{fact}
\label{fact7}For $x\in\mathcal{F}_{1}$ the sequence\/ $\left(a_{i}\left(x\right)\right)$
is necessarily unbounded. This is to say that the set $\mathcal{F}_{1}$
is contained in the complement of the set $\mathcal{B}$ of badly
approximable numbers. 
\end{fact}
The Hausdorff dimension $\dim_{H}\left(\mathcal{F}_{\alpha}\right)$
is an appropriate quantity to measure the size of the sets $\mathcal{F}_{\alpha}$.
In this paper we will give a complete analysis of the \emph{ arithmetic-geometric
scaling spectrum} \[
f\left(\alpha\right):=\dim_{H}\left(\mathcal{F}_{\alpha}\right),\qquad\alpha\in\R.\]

We already know from Fact \ref{fact5} that the maximal Hausdorff
dimension $f\left(\alpha\right)=1$ is attained for $\alpha=\alpha_{0}$
and $f$ is zero outside of $\left[0,1\right]$. Since the noble numbers
have Hausdorff dimension zero there is some evidence that $f\left(0\right)=0$.
In fact, both boundary points $0$ and $1$ will need some extra attention
concerning this analysis.

Using the Thermodynamic Formalism we will be able to express the function
$f$ on $[0,1]$ implicitly in terms of the \emph{arithmetic-geometric
pressure} \emph{function} \[
P\left(t,\beta\right):=\lim_{n\rightarrow\infty}\frac{1}{n}\lg\sum_{\omega\in\N^{n}}q_{n}\left(\omega\right)^{-2t}\prod_{i=1}^{n}\omega_{i}^{-2\beta},\quad t,\beta\in\R,\]
We shall see in Lemma \ref{lem:pressure is pressure} that the limit
defining $P$ always exists as an element of $\mathbb{R}\cup\left\{ +\infty\right\} $.
By Proposition \ref{pro:exTf(beta)} we have that for every $\beta\in\R$
there exists a unique number $t=t\left(\beta\right)$, such that $P\left(t\left(\beta\right),\beta\right)=0$.
We denote by $\beta\mapsto t\left(\beta\right)$ the \emph{arithmetic-geometric
free energy function} (see Fig. \ref{fig:The-arithmetic-geometric-fluctuation}).
For any real convex function $g$ we let $\hat{g}\colon\R\to\R\cup\left\{ \infty\right\} $
denote the \emph{Legendre transform of $g$} given by $\hat{g}\left(p\right):=\sup_{c\in\R}\left\{ cp-g(c)\right\} $,
$p\in\R$. Now we are in the position to state our main theorem.

\begin{figure}
\includegraphics[width=1\textwidth]{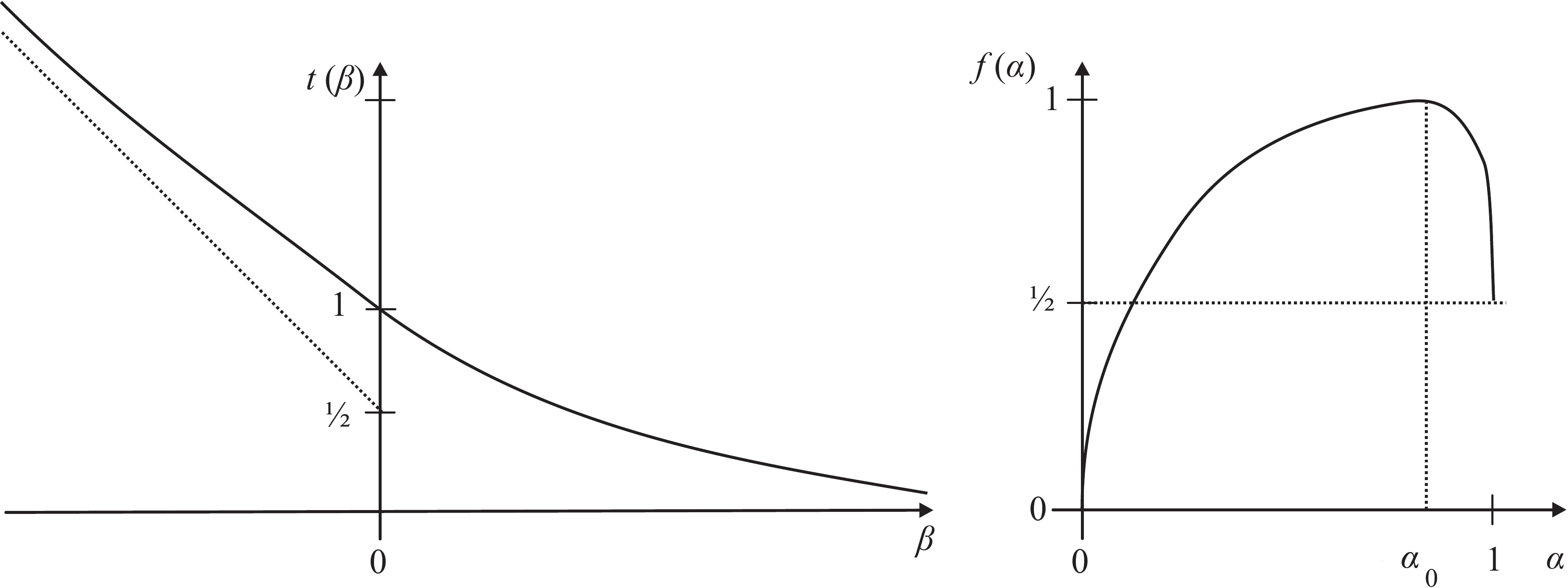}

\caption{\label{fig:The-arithmetic-geometric-fluctuation}The arithmetic-geometric
free energy function $t$ and the associated multifractal scaling
spectrum $f$.}

\end{figure}

\begin{thm}
\label{thm:maintheorem}The Hausdorff dimension spectrum (cf. Fig.
\ref{fig:The-arithmetic-geometric-fluctuation}) for the arith\-me\-tic-ge\-ome\-tric
scaling is given by \[
f\left(\alpha\right)=\max\left\{ -\hat{t}\left(-\alpha\right),0\right\} =\dim_{H}\left(\mathcal{F}_{\alpha}^{*}\right),\qquad\alpha\in\R.\]
 \textup{\emph{The function $f\big|_{\left[0,1\right]}$ is strictly
convex, continuous, and real-analytic on}}\/\textup{\emph{ $\left(0,1\right)$.
It attains its maximal value}}\/\textup{\emph{ $1$ in $\alpha_{0}=12\pi^{-2}\lg\left(2\right)\lg\left(K_{0}\right)$,
where $K_{0}$ denotes the Khintchin constant. For the boundary points
we have \[
f\left(0\right)=0,\:\: f\left(1\right)=1/2,\:\:\mbox{and }\:\lim_{\alpha\searrow0}f'\left(\alpha\right)=+\infty,\:\:\lim_{\alpha\nearrow1}f'\left(\alpha\right)=-\infty.\]
}}  
\end{thm}
The remaining part of this section is devoted to the significance
of the particular value $f\left(1\right)=\dim_{H}\left(\mathcal{F}_{1}\right)=1/2$.
We have already noticed that $\mathcal{F}_{1}$ contains the set $\mathcal{G}$
of points $x\in\mathbb{I}$ with continued fraction entries $a_{i}\left(x\right)$
tending to infinity. For this set Good proved in \cite{MR0004878}
that \begin{equation}
\dim_{H}\mathcal{G}=1/2.\label{eq:increasing sequ dim 1 2}\end{equation}
Since $\mathcal{F}_{1}\supset\mathcal{G}$, Good's results provides
us with a lower but not with an upper bound for $f\left(1\right)$.
In \cite{kessemultifractalsternbrocotMR2338129} it has been shown,
that the Hausdorff dimension of sets with large geometric scaling
coefficients are close to $1/2$, i.e. \[
\dim_{H}\left\{ x\in\left(0,1\right):\lim_{n\rightarrow\infty}\frac{2\lg q_{n}\left(x\right)}{n}=\alpha\right\} \to1/2,\quad\alpha\to\infty.\]
 Similarly, in \cite{Fan:arXiv0802.3433} we find for $\alpha>1$
and $\beta>0$, \[
\dim_{H}\left\{ x\in\left(0,1\right):\lim_{n\rightarrow\infty}\frac{\lg q_{n}\left(x\right)}{n^{\alpha}}=\beta\right\} =1/2.\]
 Ramharter has shown in \cite{MR856971} that also for every $q\in\N$
we have \[
\dim_{H}\left\{ x\in\mathbb{I}:a_{i}\left(x\right)\ge q;a_{i}\left(x\right)\neq a_{j}\left(x\right)\:\mbox{for all }i\neq j\right\} =1/2.\]
 Other results interesting in this context can be found in \cite{MR1269292},
\cite{MR1067484}, \cite{MR0311581} and \cite{MR0258778} . Furthermore,
in \cite{MR856971} we find that for $q\to\infty$ \begin{equation}
\dim_{H}\mathcal{I}_{q}=\frac{1}{2}+\mbox{O}\left(\frac{\lg\lg q}{\log q}\right),\label{eq:Ramharter}\end{equation}
 where $\mbox{O}$ denotes the usual Landau symbol, i.e. $f\left(x\right)=\mbox{O}\left(g\left(x\right)\right)$
for $x\to a$ if there exists a constant $c>0$ such that $f\left(x\right)\leq cg\left(x\right)$
for all $x$ in a neighbourhood of $a$. With some extra effort we
are able to improve (\ref{eq:Ramharter}) and obtain the precise asymptotic
of this convergence. Here $a\left(n\right)\sim b\left(n\right)$ stands
for $a\left(n\right)/b\left(n\right)\to1$ for $n\to\infty$.
\begin{prop}
\label{pro:ramharter prop}For $q\to\infty$ we have\[
\dim_{H}\mathcal{I}_{q}-\frac{1}{2}\sim\frac{1}{2}\frac{\log\log q}{\log q}.\]

\end{prop}
We would like to remark that this result is rather complementary to
the Texan conjecture (proved in \cite{KessZhu:06}), which claims
that the set of Hausdorff dimensions of bounded type continued fraction
sets is dense in the unit interval. Already Jarník observed in \cite{Jarnik:29}
that for the set of bounded continued fractions we have \[
\dim_{H}\left\{ x\in\mathbb{I}:\forall i\in\N\; a_{i}\left(x\right)\leq M\right\} =1-\mbox{O}\left(1/M\right).\]
 This was later significantly improved by Hensley, who gave a precise
asymptotic up to $\mbox{O}\left(M^{-2}\right)$ in \cite{Hensley92}.

As an interesting application of our multifractal analysis we are
able to give an asymptotic formula for the Hausdorff dimension of
$\mathcal{F}_{\alpha}$ as $\alpha$ approaches $1$. Let us write
$a\left(x\right)=\Theta(b(x))$ for $x\nearrow a$ if there exist
constants $0<c_{1}\leq c_{2}$ such that $c_{1}b\left(x\right)\leq a\left(x\right)\leq c_{2}b\left(x\right)$
for all $x$ in a (left) neighbourhood of $a$.
\begin{thm}
\label{thm:Asymp}For $\alpha\nearrow1$ we have \[
f(\alpha)=\frac{1}{2}+\Theta\left(\frac{\log\log\left(1/\left(1-\alpha\right)\right)}{\log\left(1/\left(1-\alpha\right)\right)}\right).\]
\end{thm}
\begin{rem}
Actually, the constants in the definition of $\Theta$ can be chosen
to be any $0<c_{1}<1$ and $c_{2}>2$.

In virtue of Fact \ref{fact6} there is a connection between Theorem
\ref{thm:Asymp} and Proposition \ref{pro:ramharter prop}, which
will be employed in the proof of Theorem \ref{thm:Asymp}.

We would finally like to remark that the arithmetic-geometric scaling
allows an interpretation in terms of the geodesic flow on the modular
surface. More precisely, the term $\sum\log a_{i}$ measures the homological
windings around the cusp, whereas $\log q_{n}$ stands for the total
geodesic length. The set $\mathbb{I}$ is regarded as the set of directions
for a given observation point. This connection allows a generalisation
of our formalism also to modular forms similar to \cite{kessehomologyatinfinityMR2337557}. 
\end{rem}

\section{Thermodynamic Formalism for the Gauss system}

\subsection{The Gauss system and Diophantine analysis}

The process of writing an element of $\mathbb{I}$ in its unique continued
fraction expansion can be restated by a hyperbolic dynamical system
given by the Gauss map $T$ defined in (\ref{eq:GaussMap}). The Gauss
map is conjugated to the left shift $\sigma\colon\N^{\N}\to\N^{\N}$,
$\left(\sigma\left(\omega\right)\right)_{i}=\omega_{i+1}$ for $\omega\in\N^{\N}$
via $\pi$, i.e. we have the following commutative diagram

\[
\xymatrix{\mathbb{I}\ar[r]^{T}\ar[d]_{\pi} & \mathbb{I}\ar[d]^{\pi}\\
\N^{\N}\ar[r]_{\sigma} & \N^{\N}}
\]

The Gauss system allows alternatively a representation as an infinite
conformal Iterated Function System as defined in \cite{urbanskimauldin-gdmsMR2003772}.
The system is given by the compact metric space $\left[0,1\right]$
together with the inverse branches $\Phi_{n}\colon\left[0,1\right]\to\left[0,1\right]$,
$x\mapsto\left(n+x\right)^{-1}$, $n\in\N$, of the Gauss map. Notice
that the family of maps $\left(\Phi_{n}\Phi_{m}\right)_{n,m}$ is
uniformly contracting. We are now aiming at expressing the arithmetic-geometric
scaling limit in dynamical terms. For this we introduce the two potential
functions \[
\psi\colon\N^{\N}\to\R_{0}^{-},\:\omega\mapsto-2\lg\omega_{1}\;\mbox{ and }\;\varphi\colon\N^{\N}\to\R^{-},\:\omega\mapsto-\lg\left(|T'\left(\pi^{-1}\left(\omega\right)\right)|\right),\]

where $\psi$ describes the arithmetic properties, while $\varphi$
describes the geometric properties of the continued fraction expansion.
We will equip $\N^{\N}$ with the metric $d$ given by $d(\omega,\tau):=\exp\left(-|\omega\wedge\tau|\right),$
where $|\omega\wedge\tau|$ denotes the length of the longest common
initial block of $\omega$ and $\tau$. Since $\psi$ is locally constant
we immediately see that $\psi$ is Hölder continuous with respect
to this metric. Next, we want to show that also $\varphi$ is Hölder
continuous. We start with an important observation connecting the
arithmetic and geometric properties of the continued fraction expansion.
For two sequences $\left(a_{n}\right)$, $\left(b_{n}\right)$ we
will write $a_{n}\ll b_{n}$, if $a_{n}\le Kb_{n}$ for some $K>0$
and all $n\in\N$, and if $a_{n}\ll b_{n}$ and $b_{n}\ll a_{n}$
then we write $a_{n}\asymp b_{n}$. For $\omega\in\N^{\N}$ and $n\in\N$
let $\omega|_{n}:=\left(\omega_{1},\dots,\omega_{n}\right)$ and let
$[\omega|_{n}]:=\left\{ \tau\in\N^{\N}:\tau_{1}=\omega_{1},\dots,\tau_{n}=\omega_{n}\right\} $
denote the $n$-\emph{cylinder} of $\omega$. Since we always have
$p_{n-1}\left(\omega\right)q_{n}\left(\omega\right)-p_{n}\left(\omega\right)q_{n-1}\left(\omega\right)=\left(-1\right)^{n}$
and \[
\pi^{-1}\left(\left[\omega|_{n}\right]\right):=\begin{cases}
\vspace{2pt}\left(\frac{p_{n}\left(\omega\right)}{q_{n}\left(\omega\right)},\frac{p_{n}\left(\omega\right)+p_{n-1}\left(\omega\right)}{q_{n}\left(\omega\right)+q_{n-1}\left(\omega\right)}\right)\cap\mathbb{I} & \mbox{for }n\,\mbox{ even,}\\
\left(\frac{p_{n}\left(\omega\right)+p_{n-1}\left(\omega\right)}{q_{n}\left(\omega\right)+q_{n-1}\left(\omega\right)},\frac{p_{n}\left(\omega\right)}{q_{n}\left(\omega\right)}\right)\cap\mathbb{I} & \mbox{for }n\,\mbox{ odd},\end{cases}\]
 it follows that\[
\diam\left(\pi^{-1}[\omega_{|n}]\right)=q_{n}\left(\omega\right)^{-1}\left(q_{n}\left(\omega\right)+q_{n-1}\left(\omega\right)\right)^{-1}=q_{n}\left(\omega\right)^{-2}\left(1+\frac{q_{n-1}\left(\omega\right)}{q_{n}\left(\omega\right)}\right)^{-1}\]
for all $n\in\N$ (see e.g. \cite{Khinchin:56}). This gives \begin{equation}
\diam\left(\pi^{-1}[\omega_{|n}]\right)\asymp\frac{1}{q_{n}^{2}\left(\omega\right)}\label{eq:diam comparable 1 qn2}\end{equation}
where the constants are independent of $\omega\in\N^{\N}$. With $f_{n}$
denoting the $n$-th Fibonacci number we have that $q_{n}\left(\omega\right)\geq f_{n}\gg\gamma^{n}$,
where $\gamma:=\left(\sqrt{5}+1\right)/2$ refers to the \emph{Golden
Mean}. Fix $v,w\in[\omega]$ for some $\omega\in\N^{n}$. Then, using
(\ref{eq:diam comparable 1 qn2}), we get \begin{eqnarray*}
\left|\varphi\left(v\right)-\varphi\left(w\right)\right| & = & 2\left|\log\left(\pi^{-1}v\right)-\log\left(\pi^{-1}w\right)\right|=2\left|\log\left(1+\frac{\pi^{-1}w-\pi^{-1}v}{\pi^{-1}v}\right)\right|\\
 & \ll & \frac{q_{n}\left(\omega\right)}{p_{n}\left(\omega\right)}q_{n}\left(\omega\right)^{-2}\ll d\left(v,w\right)^{2\log\gamma},\end{eqnarray*}
which proves the Hölder continuity of $\varphi$. From this we also
deduce the so-called \emph{bounded distortion property}\begin{equation}
\frac{\left|\phi_{\omega|_{n}}'\left(x\right)\right|}{\left|\phi_{\omega|_{n}}'\left(y\right)\right|}\asymp1,\label{eq:boundedDistortion1}\end{equation}
where $\phi_{\omega|_{n}}:=\phi_{\omega_{1}}\circ\cdots\circ\phi_{\omega_{n}}$
and the constants are independent of $\omega\in\N^{\N}$ and $x,y\in\mathbb{I}$.
The bounded distortion property in particular implies $\left|\phi_{\omega|_{n}}'\left(x\right)\right|\asymp\diam\left(\pi^{-1}[\omega|_{n}]\right).$
Using this it is possible to compare the diameters of cylinder sets
with orbit sums $S_{n}\phi:=\sum_{k=0}^{n-1}\phi\circ\sigma^{k}$
with respect to the geometric potential $\varphi$ under iterations
of the shift map $\sigma$ . In fact, by the chain rule and (\ref{eq:boundedDistortion1})
we have uniformly for $\omega\in\N^{\N}$ and $\tau\in\left[\omega|_{n}\right]$ 

\begin{equation}
\exp S_{n}\varphi\left(\tau\right)\asymp\diam\left(\pi^{-1}[\omega|_{n}]\right).\label{eq:boundeddistortion}\end{equation}

\subsection{Topological pressure }

The topological pressure $\mathfrak{P}\left(t\varphi+\beta\psi\right)$
of the potential $t\varphi+\beta\psi$ for $t,\beta\in\R$ is defined
as \[
\mathfrak{P}\left(t\varphi+\beta\psi\right):=\lim_{n\rightarrow\infty}\frac{1}{n}\lg\sum_{\omega\in\mathbb{N}^{n}}\exp\sup_{\tau\in[\omega]}\left(S_{n}t\varphi+\beta\psi\right)(\tau).\]
By a standard argument involving sub-additivity the above limit always
exists.

The next lemma shows, that the set $\mathcal{F}_{\alpha}$ can be
characterized by the potentials $\varphi$ and $\psi$ and that the
arithmetic-geometric pressure $P$$\left(t,\beta\right)$ agrees with
$\mathfrak{P}\left(t\varphi+\beta\psi\right)$, $t,\beta\in\R$.
\begin{lem}
\label{lem:pressure is pressure} For $\alpha\in\R$ and $x\in\mathbb{I}$
we have \[
\lim_{n\rightarrow\infty}\frac{S_{n}\psi\left(\pi\left(x\right)\right)}{S_{n}\varphi\left(\pi\left(x\right)\right)}=\alpha\iff\lim_{n\rightarrow\infty}\frac{\log\prod_{j=1}^{n}a_{j}\left(x\right)}{\log q_{n}\left(x\right)}=\alpha\]
 and \[
P\left(t,\beta\right)=\mathfrak{P}\left(t\varphi+\beta\psi\right),\qquad t,\beta\in\R.\]
\end{lem}
\begin{proof}
By (\ref{eq:diam comparable 1 qn2}) and (\ref{eq:boundeddistortion})
there exist constants $C_{1},C_{2}>0$, such that for all $\omega\in\N^{\N}$
and $n\in\N$ we have\begin{equation}
-2\lg q_{n}\left(\omega\right)+\lg C_{1}\le S_{n}\varphi\left(\omega\right)\le-2\lg q_{n}\left(\omega\right)+\lg C_{2}.\label{eq:4}\end{equation}
 Dividing this inequality by $S_{n}\psi\left(\omega\right)=-2\lg\left(\prod_{j=1}^{n}\omega_{j}\right)$
and using the fact that $q_{n}\left(\omega\right)$ tends to infinity
for $n\rightarrow\infty$ proves the first assertion.

To prove the second claim notice that by (\ref{eq:diam comparable 1 qn2})
and the definition of $\psi$ we have \begin{eqnarray*}
\sum_{\omega\in\N^{n}}\exp\sup_{\tau\in[\omega]}\left(S_{n}t\varphi+\beta\psi\right)(\tau) & \asymp & \sum_{\omega\in\N^{n}}q_{n}^{-2t}\left(\omega\right)\left(\prod_{j=1}^{n}\omega_{j}\right)^{-2\beta}.\end{eqnarray*}
 Taking logarithms and dividing by $n$ again proves the claim. \end{proof}
\begin{lem}
\label{lem:pressureFiniteIFF} We have \begin{equation}
P\left(t,\beta\right)<\infty\iff2\left(t+\beta\right)>1.\label{eq:pressure finite if and onl if}\end{equation}
\end{lem}
\begin{proof}
Using (\ref{eq:qn versus prod ai einfach}) we have on the one hand
for $t\le0$ \[
\zeta\left(2\left(t+\beta\right)\right)^{n}\ll\sum_{\omega\in\N^{n}}q_{n}^{-2t}\left(\omega\right)\left(\prod_{j=1}^{n}\omega_{j}\right)^{-2\beta}\ll2^{-2nt}\zeta\left(2\left(t+\beta\right)\right)^{n},\]
 where $\zeta$ denotes the Riemann zeta function, which is singular
in $1$. On the other hand for $t>0$ we have \[
2^{-2nt}\zeta\left(2\left(t+\beta\right)\right)^{n}\ll\sum_{\omega\in\N^{n}}q_{n}^{-2t}\left(\omega\right)\left(\prod_{j=1}^{n}\omega_{j}\right)^{-2\beta}\ll\zeta\left(2\left(t+\beta\right)\right)^{n}.\]
Taking logarithms and dividing by $n$ then gives in both cases the
asserted equivalence. 
\end{proof}
For later use we will need a refined lower estimate for $q_{n}$,
which also relies on the recursion formula (\ref{eq:pnqn recursion})
for $q_{n}$.
\begin{lem}
\label{lem:qn vs prod ai better lower bound}For $\omega=\left(\omega_{1},\omega_{2},\dots\right)\in\N^{\N}$
and $n\in\N$ we have \[
\omega_{1}\prod_{i=2}^{n}\omega_{i}\left(1+\frac{1}{\omega_{i}\left(\omega_{i-1}+1\right)}\right)\le q_{n}\left(\omega\right)\]
\end{lem}
\begin{proof}
The proof is by means of induction. For $n=1$ we have $q_{1}\left(\omega\right)=\omega_{1}$.
For $n>1$ we have by the recursion formula (\ref{eq:pnqn recursion})
that \begin{equation}
q_{n}\left(\omega\right)=\omega_{n}q_{n-1}\left(\omega\right)\left(1+\frac{q_{n-2}\left(\omega\right)}{\omega_{n}q_{n-1}\left(\omega\right)}\right)\label{eq:1}\end{equation}
 and also \begin{equation}
\frac{q_{n-1}\left(\omega\right)}{q_{n-2}\left(\omega\right)}=\omega_{n-1}+\frac{q_{n-3}\left(\omega\right)}{q_{n-2}\left(\omega\right)}\le\omega_{n-1}+1.\label{eq:2}\end{equation}
Combining (\ref{eq:2}) and (\ref{eq:1}) gives \[
\omega_{n}\left(1+\frac{1}{\omega_{n}\left(\omega_{n-1}+1\right)}\right)q_{n-1}\left(\omega\right)\le q_{n}\left(\omega\right),\]
 which proves the inductive step. 
\end{proof}
The next proposition gives bounds for the pressure $P\left(t,\beta\right)$,
which will be essential for the discussion of the boundary points
of the multifractal spectrum.
\begin{prop}
\label{pro:p estimates}We have for $t\ge0$

\[
\lg\negmedspace\left(\sum_{k\in\N}\negmedspace(k+1)^{-2t}k^{-2\beta}\right)\le P\left(t,\beta\right)\le\frac{1}{2}\lg\negmedspace\left(\sum_{\left(k,l\right)\in\N^{2}}\negmedspace\negmedspace\left(kl\right)^{-2\left(t+\beta\right)}\left(1+\frac{1}{k\left(l+1\right)}\right)^{-2t}\right).\]
\end{prop}
\begin{proof}
Using the fact that $q_{n}\left(\omega\right)\leq\prod_{k=1}^{n}\left(\omega_{k}+1\right)$
we obtain a a lower bound \begin{eqnarray*}
\sum_{\omega\in\N^{n}}q_{n}\left(\omega\right)^{-2t}\prod_{i=1}^{n}a_{i}^{-2\beta} & \ge & \sum_{\omega\in\N^{n}}\prod_{i=1}^{n}\left(\omega_{i}+1\right)^{-2t}\omega_{i}^{-2\beta}\\
 & = & \left(\sum_{k\in\N}\left(k+1\right)^{-2t}k^{-2\beta}\right)^{n},\end{eqnarray*}
 by rearranging the series. Taking logarithm and dividing by $n$
shows \[
P\left(t,\beta\right)\ge\lg\left(\sum_{k\in\N}(k+1)^{-2t}k^{-2\beta}\right).\]
 For the upper bound we use Lemma \ref{lem:qn vs prod ai better lower bound}
to conclude \begin{eqnarray*}
\sum_{\omega\in\N^{k}}q_{k}\left(\omega\right)^{-2t}\prod_{i=1}^{k}\omega_{i}^{-2\beta} & \le & \sum_{\omega\in\N^{k}}\left(\omega_{1}\prod_{i=2}^{k}\omega_{i}\left(1+\frac{1}{\omega_{i}\left(\omega_{i-1}+1\right)}\right)\right)^{-2t}\prod_{i=1}^{k}\omega_{i}^{-2\beta}\\
 & = & \sum_{\omega\in\N^{k}}\prod_{i=2}^{k}\left(1+\frac{1}{\omega_{i}\left(\omega_{i-1}+1\right)}\right)^{-2t}\prod_{i=1}^{k}\omega_{i}^{-2\left(t+\beta\right)}.\end{eqnarray*}
 Now, we only consider even $k=2n$. Since $\left(1+\left(\omega_{i}\left(\omega_{i-1}+1\right)\right)^{-1}\right)^{-2t}<1$
for all $i\geq2$, we find an upper bound by omitting all terms with
odd indices $i$ in the product $\prod_{i=2}^{2n}\left(1+\left(\omega_{i}\left(\omega_{i-1}+1\right)\right)^{-1}\right)^{-2t}$.
Using this and rearranging the series we get \begin{eqnarray*}
\sum_{\omega\in\N^{k}}q_{k}\left(\omega\right)^{-2t}\prod_{i=1}^{k}\omega_{i}^{-2\beta} & \le & \sum_{\omega\in\N^{k}}\prod_{i=1}^{n}\left(1+\frac{1}{\omega_{2i}\left(\omega_{2i-1}+1\right)}\right)^{-2t}\left(\omega_{2i-1}\omega_{2i}\right)^{-2\left(t+\beta\right)}.\end{eqnarray*}
 \begin{eqnarray*}
 & = & \sumsum_{\left(\omega_{1},\omega_{2}\right),\left(\omega_{3},\omega_{4}\right),\ldots,\left(\omega_{2n-1},\omega_{2n}\right)\in\N^{2}}\;\prod_{i=1}^{n}\left(1+\frac{1}{\omega_{2i}\left(\omega_{2i-1}+1\right)}\right)^{-2t}\;\left(\omega_{2i-1}\omega_{2i}\right)^{-2\left(t+\beta\right)}\;.\\
 & = & \left(\sum_{\left(k,l\right)\in\N^{2}}\left(kl\right)^{-2\left(t+\beta\right)}\left(1+\frac{1}{k\left(l+1\right)}\right)^{-2t}\right)^{\frac{n}{2}}.\end{eqnarray*}
 Taking logarithm and dividing by $n$ gives \[
P\left(t,\beta\right)\le\frac{1}{2}\lg\left(\sum_{\left(k,l\right)\in\N^{2}}\left(kl\right)^{-2\left(t+\beta\right)}\left(1+\frac{1}{k\left(l+1\right)}\right)^{-2t}\right).\]
\end{proof}
\begin{rem}
A straight forward calculation shows that for $t=0$ and $\beta>1/2$
we have \[
P\left(0,\beta\right)=\lim_{n\rightarrow\infty}\frac{1}{n}\lg\sum_{\omega\in\N^{n}}\prod_{i=1}^{n}a_{i}\left(\omega\right)^{-2\beta}=\lim_{n\rightarrow\infty}\frac{1}{n}\lg\left(\prod_{k=1}^{n}k^{-2\beta}\right)^{n}=\lg\left(\zeta\left(2\beta\right)\right).\]
 This value coincides for $t=0$ with the upper bound in Proposition
\ref{pro:p estimates} since \[
\frac{1}{2}\lg\left(\sum_{\left(k,l\right)\in\N^{2}}\left(kl\right)^{-2\beta}\right)=\frac{1}{2}\lg\left(\sum_{k\in\N}k^{-2\beta}\sum_{l\in\N}l^{-2\beta}\right)=\lg\left(\zeta\left(2\beta\right)\right).\]

\end{rem}

\subsection{\label{sub:Proofs-of-facts}Proof of facts}

With the results obtain in the previous subsections we are in the
position to give the proofs of the Facts \ref{fact1} to \ref{fact7}
stated in the introduction.

\emph{Proof of Facts \ref{fact1} and \ref{fact2}.} These facts are
immediate consequences of the first inequality in (\ref{eq:qn versus prod ai einfach}).
\qed

\emph{Proof of Fact \ref{fact3}.} First notice that $\pi^{-1}\left(k,k,\dots\right)$,
$k\in\N$, is a fixed point of the Gauss map $T$ and hence invariant
under $x\mapsto1/x-k$. This implies $\pi^{-1}\left(k,k,\dots\right)=-k/2+\sqrt{k^{2}/4+1}$.
Using (\ref{eq:4}) in the proof of Lemma \ref{lem:pressure is pressure}
gives for $q_{n}:=q_{n}\left(\left(k,k,\ldots\right)\right)$ \[
\lim_{n\to\infty}\frac{\log q_{n}}{n}=-\frac{1}{2}\varphi\left(\left(k,k,\ldots\right)\right)=-\lg\left(\pi^{-1}\left(k,k,\ldots\right)\right)=-\lg\left(-k/2+\sqrt{k^{2}/4+1}\right).\]
 From this the claims follow. \qed

\emph{Proof of Fact \ref{fact4}.} Using (\ref{eq:qn versus prod ai einfach})
we have\begin{eqnarray*}
\frac{\sum_{i=1}^{n}\lg a_{i}\left(x\right)}{\lg q_{n}\left(x\right)} & \ge & \frac{\sum_{i=1}^{n}\lg a_{i}\left(x\right)}{n\lg2+\sum_{i=1}^{n}\lg a_{i}\left(x\right)}\\
 & = & \left(\frac{\lg2}{n^{-1}\sum_{i=1}^{n}\lg a_{i}\left(x\right)}+1\right)^{-1}\to1,\qquad n\to\infty.\end{eqnarray*}
 Here we have used that the Cesàro mean of $\lg a_{i}\left(x\right)$
tends to infinity. \qed

\emph{Proof of Fact} \emph{\ref{fact5}.} Let us consider the ergodic
dynamical system $\left(\mathbb{I},T,\lambda_{g}\right)$ where $d\lambda_{g}\left(x\right):=\left(\log\left(2\right)\left(1+x\right)\right)^{-1}d\lambda\left(x\right)$
denotes the famous Gauss measure. By the Ergodic Theorem we have $\lambda_{g}$-a.e.
and consequently $\lambda$-a.e. \[
\lim_{n}\frac{S_{n}\psi}{n}:=\int\psi\, d\lambda_{g}=\frac{-2}{\log2}\sum_{k=1}^{\infty}\log\left(k\right)\log\left(1+\frac{1}{k\left(k+2\right)}\right)=-2\log\left(K_{0}\right)\]
as well as\[
\lim_{n}\frac{S_{n}\phi}{n}:=\int\phi\, d\lambda_{g}=\frac{2}{\log2}\int\frac{\log\left(x\right)}{\left(1+x\right)}\, d\lambda=\frac{-\zeta\left(2\right)}{\log2}=\frac{-\pi^{2}}{6\log2}.\]
From this the fact follows immediately. \qed

\emph{Proof of Fact \ref{fact6}.} Using the inequality $q_{n}\left(x\right)\le\prod_{i=1}^{n}\left(a_{i}\left(x\right)+1/q\right)$
for $x\in\mathcal{I}_{q}$ and $n\in\N$ we have\\
\\
${\displaystyle \frac{\sum_{i=1}^{n}\lg a_{i}\left(x\right)}{\lg q_{n}\left(x\right)}}$\begin{eqnarray*}
 & \ge & \frac{\sum_{i=1}^{n}\lg a_{i}\left(x\right)}{\sum_{i=1}^{n}\lg\left(a_{i}\left(x\right)+1/q\right)}=\frac{\sum_{i=1}^{n}\lg a_{i}\left(x\right)}{\sum_{i=1}^{n}\lg a_{i}\left(x\right)+\sum_{i=1}^{n}\lg\left(1+1/qa_{i}\left(x\right)\right)}\\
 & = & \frac{\sum_{i=1}^{n}\lg a_{i}\left(x\right)}{\sum_{i=1}^{n}\lg a_{i}\left(x\right)+\sum_{i=1}^{n}\lg\left(1+1/qa_{i}\left(x\right)\right)}\geq\frac{\sum_{i=1}^{n}\lg a_{i}\left(x\right)}{\sum_{i=1}^{n}\lg a_{i}\left(x\right)+n/q^{2}}\\
 & = & \left(1+\frac{n}{q^{2}\sum_{i=1}^{n}\lg a_{i}\left(x\right)}\right)^{-1}\geq\left(1+\left(q^{2}\log\left(q\right)\right)^{-1}\right)^{-1}\\
 & = & 1-\frac{\left(q^{2}\log\left(q\right)\right)^{-1}}{1+\left(q^{2}\log\left(q\right)\right)^{-1}}\geq1-\left(q^{2}\log\left(q\right)\right)^{-1}.\hspace{3cm}\qed\end{eqnarray*}

\emph{Proof of Fact \ref{fact7}.} We show that $\mathcal{B}\subset\mathbb{I}\setminus\mathcal{F}_{1}$.
Let us assume that for an element $x\in\mathcal{B}$ the sequence
$\left(a_{i}\left(x\right)\right)_{i\in\N}$ is bounded by $M\ge2$.
Then using the lower bound for $q_{n}$ provided in Lemma \ref{lem:qn vs prod ai better lower bound}
and the fact that $\log\left(1+t\right)\ge t\cdot\log(2),$ $t\in\left[0,1\right]$
we calculate for $n\geq2$\begin{eqnarray*}
\frac{\sum_{i=1}^{n}\lg a_{i}\left(x\right)}{\lg q_{n}\left(x\right)} & \le & \frac{\sum_{i=1}^{n}\lg a_{i}\left(x\right)}{\sum_{i=1}^{n}\lg a_{i}\left(x\right)+\sum_{i=2}^{n}\lg\left(1+\left(a_{i}\left(x\right)\left(a_{i-1}\left(x\right)+1\right)\right)^{-1}\right)}\\
 & \le & \left(1+\frac{\sum_{i=2}^{n}\log\left(2\right)/a_{i}\left(x\right)\left(a_{i-1}\left(x\right)+1\right)}{\sum_{i=1}^{n}\lg a_{i}\left(x\right)}\right)^{-1}\\
 & \leq & \left(1+\frac{\log2\left(n-1\right)}{M\left(M+1\right)n\lg M}\right)^{-1}\\
 & \leq & \left(1+\frac{\log2}{2M\left(M+1\right)\lg M}\right)^{-1}<1.\end{eqnarray*}
Since the left hand side is bounded away from $1$ by a constant only
depending on $M$ the fact follows. \qed

\subsection{Gibbs states}

Let us recall some basic facts about Gibbs states taken from \cite{urbanskimauldin-gdmsMR2003772}.
For a continuous function $f\colon\N^{\N}\to\R$ a Borel probability
measure $m$ on $\N^{\N}$ is called a \emph{Gibbs state for $f$,}
if there exists a constant $Q\ge1$ such that for every \emph{$n\in\N$,}
$\omega\in\N^{n}$ and \emph{$\tau\in[\omega]$} we have\begin{equation}
Q^{-1}\le\frac{m([\omega])}{\exp(S_{n}f(\tau)-n\mathfrak{P}\left(f\right))}\le Q.\label{eq:gibbsequ}\end{equation}
 If in addition the measure $m$ is $\sigma$-invariant then $m$
is called an \emph{invariant Gibbs state} for $f$.

Also the concept of the metric entropy will by crucial. Recall that
in our situation for a $\sigma$-invariant measure $\mu$ the \emph{metric
entropy} is given by \[
h_{\mu}:=\lim_{n}\frac{1}{n}\sum_{\omega\in\N^{n}}\mu\left(\left[\omega\right]\right)\log\left(\mu\left(\left[\omega\right]\right)\right),\]
where as usual we set $0\cdot\log0=0$. Note, that the above limit
always exists (see e.g. \cite{walters-ergodictheoryMR648108}). 

The next proposition states the key result of the Thermodynamic Formalism
in our context, that is the existence and uniqueness of equilibrium
measures for the Hölder continuous and summable potential $t\varphi+\beta\psi$
with $2\left(t+\beta\right)>1$, i.e. $\sum_{k\in\N}\exp\sup_{\tau\in\left[k\right]}\left(t\varphi+\beta\psi\right)\left(\tau\right)<\infty$.
For a proof we refer to \cite{urbanskimauldin-gdmsMR2003772} (see
e.g. \cite{bowenequilibriumMR0442989} for a classical version valid
for compact state spaces).
\begin{prop}
\label{pro:existence of invariant gibbs}For each\, $\left(t,\beta\right)\in\R^{2}$
such that\, $2\left(t+\beta\right)>1$ there exists a unique invariant
Gibbs state $\mu_{t\varphi+\beta\psi}$ for the potential $t\varphi+\beta\psi$,
which is ergodic and an equilibrium state for the potential, i.e.
$P\left(t,\beta\right)=h_{\mu_{t\varphi+\beta\psi}}+\int t\varphi+\beta\psi\, d\mu_{t\varphi+\beta\psi}$.
\end{prop}
We close this subsection with a technical lemma needed for the proof
of Proposition \ref{pro:exTf(beta)}.
\begin{lem}
\label{lem:integralfinitepropertychecked} For each\/ $\left(t,\beta\right)\in\R^{2}$
such that\/ $2\left(t+\beta\right)>1$ we have $\varphi,\psi\in\mathcal{L}\left(\mu_{t\varphi+\beta\psi}\right)$. \end{lem}
\begin{proof}
Since we have $|\psi|\le|\varphi|$ it suffices to show $\varphi\in\mathcal{L}\left(\mu_{t\varphi+\beta\psi}\right)$.
We have \begin{eqnarray*}
\int|\varphi|\, d\mu_{t\varphi+\beta\psi} & \le & \sum_{i\in\N}\sup\left(|\varphi|_{|\left[i\right]}\right)\mu_{t\varphi+\beta\psi}\left(\left[i\right]\right).\end{eqnarray*}
Observing $\sup\left(|\varphi|_{|\left[i\right]}\right)\le\lg\left(i+1\right)$
and using the Gibbs property (\ref{eq:gibbsequ}) for $\mu_{t\varphi+\beta\psi}$
we have \begin{eqnarray*}
\sum_{i\in\N}\sup\left(|\varphi|_{|\left[i\right]}\right)\mu_{t\varphi+\beta\psi}\left(\left[i\right]\right) & \ll & \sum_{i\in\N}\lg\left(i+1\right)\exp\sup\left(t\varphi+\beta\psi_{|\left[i\right]}\right)\\
 & \ll & \sum_{i\in\N}\lg\left(i+1\right)\frac{1}{i^{2\left(t+\beta\right)}}<\infty.\end{eqnarray*}
 
\end{proof}

\subsection{The arithmetic-geometric free energy }

To guarantee that the free energy function is non-linear and hence
the multifractal spectrum is non-trivial we need the following observation.
\begin{lem}
\label{lem:noncohomologous}The potentials $\varphi$ and $\psi$
are linear independent in the cohomology class of bounded Hölder continuous
functions, i.e. for every bounded Hölder continuous function $u$
satisfying \textup{$\alpha\varphi+\beta\psi=u-u\circ\sigma$ we have
$\alpha=\beta=0$.}  \end{lem}
\begin{proof}
Suppose there exists a bounded Hölder continuous function $u:\N^{\N}\to\R$,
such that \[
\alpha\varphi+\beta\psi=u-u\circ\sigma.\]
 Since $u$ is bounded, there exists $C<\infty$, such that for all
$n\in\N$ \[
\left\Vert S_{n}\left(u-u\circ\sigma\right)\right\Vert =\left\Vert u-u\circ\sigma^{n}\right\Vert <C\]
 where $\left\Vert \cdot\right\Vert $ denotes the uniform norm on
the space of bounded continuous functions. This implies for all $n\in\N$
\begin{equation}
\left\Vert \alpha S_{n}\varphi+\beta S_{n}\psi\right\Vert <C.\label{eq:cohomologyestimate}\end{equation}
 For $\omega=\left(1,1,1,\dots\right)$ we have $\beta S_{n}\psi(\omega)+\alpha S_{n}\varphi\left(\omega\right)=\alpha S_{n}\varphi\left(\omega\right)=2n\alpha\log\gamma$
for all $n\in\N$. This stays bounded only for $\alpha=0$. Furthermore,
for $\omega=\left(2,2,2,\ldots\right)$ we have $\beta S_{n}\psi=2n\beta\lg(2)$
for all $n\in\N$. Again this stays bounded only if also $\beta=0$. \end{proof}
\begin{prop}
\label{pro:exTf(beta)}For each $\beta\in\R$ there exists a unique
number $t\left(\beta\right)$ such that \begin{equation}
P\left(t\left(\beta\right),\beta\right)=0.\label{eq:pressure equation in gauss}\end{equation}
 The arithmetic-geometric free energy function $t$ defined in this
way is real-analytic and strictly convex, and we have \begin{equation}
t'\left(\beta\right)=-\frac{\int\psi\, d\mu_{\beta}}{\int\varphi\, d\mu_{\beta}}<0,\label{eq:7}\end{equation}
 where $\mu_{\beta}$ denotes the unique invariant Gibbs state for
$t\left(\beta\right)\varphi+\beta\psi$. \end{prop}
\begin{proof}
By \cite[Theorem 2.6.12]{urbanskimauldin-gdmsMR2003772} we know that
the pressure $P$ is real-analytic on $\left\{ \left(t,\beta\right)\in\R^{2}:P\left(t,\beta\right)<\infty\right\} $.
Hence by Lemma \ref{lem:pressureFiniteIFF}, $P$ is real-analytic
precisely on$\left\{ 2\left(t+\beta\right)>1\right\} $. By \cite[Proposition 2.6.13]{urbanskimauldin-gdmsMR2003772}
the partial derivatives can be expressed as integrals, i.e. \[
\frac{\partial}{\partial\beta}P\left(t,\beta\right)=\int\psi\, d\mu_{t\varphi+\beta\psi},\;\frac{\partial}{\partial t}P\left(t,\beta\right)=\int\varphi\, d\mu_{t\varphi+\beta\psi},\]
where Lemma \ref{lem:integralfinitepropertychecked} assures that
$\varphi,\psi\in\mathcal{L}\left(\mu_{t\varphi+\beta\psi}\right)$
for all $\left(t,\beta\right)$ with $2\left(t+\beta\right)>1$. Since
$\mu$ is ergodic we have \[
\int\varphi\, d\mu\le\sup_{x\in\N^{\N}}\limsup_{n}\frac{S_{n}\varphi\left(x\right)}{n}=\varphi\left(\left(1,1,1,\ldots\right)\right)=-2\log\gamma,\]
 where again $\gamma$ denotes the Golden Mean. Consequently, \begin{equation}
\frac{\partial}{\partial t}P\left(t,\beta\right)=\int\varphi\, d\mu_{t\varphi+\beta\psi}\le-2\log\gamma<0\label{eq:pressure bounded away from zero}\end{equation}
 is bounded away from zero.

Now let $\beta\in\R$. By Lemma \ref{lem:pressureFiniteIFF} we have
that $P\left(t,\beta\right)<\infty$, if and only if $2\left(t+\beta\right)>1$.
Also, since $\lim_{t\searrow\frac{1}{2}-\beta}P\left(t,\beta\right)=\infty$
we find $t_{0}\in\R$ such that $0<P\left(t_{0},\beta\right)<\infty$.
By (\ref{eq:pressure bounded away from zero}) we conclude that there
exists a unique $t=t\left(\beta\right)$ with $P\left(t\left(\beta\right),\beta\right)=0$.
By the implicit function theorem and (\ref{eq:pressure bounded away from zero})
we have that the function $t$ is real-analytic and\begin{equation}
t'\left(\beta\right)=-\frac{\frac{\partial}{\partial\beta}P\left(t\left(\beta\right),\beta\right)}{\frac{\partial}{\partial t}P\left(t\left(\beta\right),\beta\right)}=-\frac{\int\psi\, d\mu_{\beta}}{\int\varphi\, d\mu_{\beta}}<0.\label{eq:6}\end{equation}
 Concerning the strict convexity of $t$ we follow \cite{kessehomologyatinfinityMR2337557}.
Observe that \[
-\frac{\partial}{\partial t}P\left(t\left(\beta\right),\beta\right)t''\left(\beta\right)=\sigma_{\beta}^{2}\left(t'\left(\beta\right)\varphi+\psi\right)\ge0,\]
 where\\
 \\
 $\sigma_{\beta}^{2}\left(t'\left(\beta\right)\varphi+\psi\right)=$\[
\lim_{n\longrightarrow\infty}\frac{1}{n}\int\left(S_{n}\left(\left(t'\left(\beta\right)\varphi+\psi\right)-\int\left(t'\left(\beta\right)\varphi+\psi\right)\, d\mu_{\beta}\right)\right)^{2}\, d\mu_{\beta}\]
is the asymptotic variance of $S_{n}\left(t'\left(\beta\right)\varphi+\psi\right)$
with respect to the invariant Gibbs measure $\mu_{\beta}$. Since
$\int\left(t'\left(\beta\right)\varphi+\psi\right)\, d\mu_{\beta}=0$
by (\ref{eq:6}) we can conclude by (\cite[Lemma 4.88]{urbanskimauldin-gdmsMR2003772})
that $\sigma_{\beta}^{2}\left(t'\left(\beta\right)\varphi+\psi\right)>0$,
since $\varphi$ and $\psi$ are elements of $\mathcal{L}^{2}\left(\mu_{\beta}\right)$
by Lemma \ref{lem:integralfinitepropertychecked} and are linearly
independent in the cohomology class of bounded Hölder continuous functions
by Lemma \ref{lem:noncohomologous}. 
\end{proof}
The following lemma will be crucial for the the asymptotic properties
of $f$ in $1$ and will be used in the proofs of the main theorems
in Section \ref{sec:Multifractal-Analysis}.
\begin{lem}
\label{lem:tAsymptoticFor-infty}For all\/ $0<\epsilon<1/2$ we have
\[
t\left(\beta\left(\epsilon\right)\right)<1/2-\beta\left(\epsilon\right)+\frac{\epsilon}{2}\quad\mbox{with }\;\beta\left(\epsilon\right):=\frac{3}{\log(2)}\log\left(\epsilon\right)\left(\frac{\epsilon}{3}\right)^{-4/\epsilon}.\]
\end{lem}
\begin{proof}
Let us assume on the contrary that there exists $0<\epsilon<1/2$
such that $t\left(\beta\left(\epsilon\right)\right)\ge1/2-\beta\left(\epsilon\right)+\epsilon/2$.
This implies $-2\left(t\left(\beta\left(\epsilon\right)\right)+\beta\left(\epsilon\right)\right)\le-\left(1+\epsilon\right)$
as well as $-2t\left(\beta\left(\epsilon\right)\right)\le2\beta\left(\epsilon\right)-\epsilon-1\le2\beta\left(\epsilon\right)$.
Consequently, by definition of $t$ and Proposition \ref{pro:p estimates}
we would have \begin{eqnarray*}
0 & = & P\left(t\left(\beta\left(\epsilon\right)\right),\beta\left(\epsilon\right)\right)\\
 & \le & \frac{1}{2}\lg\sum_{\left(k,l\right)\in\N^{2}}\left(kl\right)^{-2\left(t\left(\beta\left(\epsilon\right)\right)+\beta\left(\epsilon\right)\right)}\left(1+\frac{1}{k\left(l+1\right)}\right)^{-2t\left(\beta\left(\epsilon\right)\right)}\\
 & \le & \frac{1}{2}\lg\sum_{\left(k,l\right)\in\N^{2}}\left(kl\right)^{-\left(1+\epsilon\right)}\left(1+\frac{1}{k\left(l+1\right)}\right)^{2\beta\left(\epsilon\right)}.\end{eqnarray*}
 To obtain a contradiction we will show that \[
\sum_{\left(k,l\right)\in\N^{2}}\left(kl\right)^{-\left(1+\epsilon\right)}\left(1+\frac{1}{k\left(l+1\right)}\right)^{2\beta\left(\epsilon\right)}<1.\]
 In fact, for $N\left(\epsilon\right):=\left(\epsilon/3\right)^{-2/\epsilon}$
we have for $0<\epsilon<1/2$ that 
\begin{itemize}
\item [(A)] ${\displaystyle {\displaystyle \sums_{k>N\left(\epsilon\right)\textrm{ or }l>N\left(\epsilon\right)}\left(kl\right)^{-\left(1+\epsilon\right)}\left(1+\frac{1}{k\left(l+1\right)}\right)^{2\beta\left(\epsilon\right)}}<1/2}$
~ ~ ~and 
\item [(B)]${\displaystyle \sums_{k\le N\left(\epsilon\right)\textrm{ and }l\le N\left(\epsilon\right)}\left(kl\right)^{-\left(1+\epsilon\right)}\left(1+\frac{1}{k\left(l+1\right)}\right)^{2\beta\left(\epsilon\right)}}<1/2$. 
\end{itemize}
To prove (A) notice that\\
 \\
 ${\displaystyle \sums_{k>N\left(\epsilon\right)\textrm{ or }l>N\left(\epsilon\right)}\left(kl\right)^{-\left(1+\epsilon\right)}\left(1+\frac{1}{k\left(l+1\right)}\right)^{2\beta\left(\epsilon\right)}}$\begin{eqnarray*}
 & \le & \sum_{k>N\left(\epsilon\right)}\sum_{l\in\N}\left(kl\right)^{-\left(1+\epsilon\right)}+\sum_{l>N\left(\epsilon\right)}\sum_{k\in\N}\left(kl\right)^{-\left(1+\epsilon\right)}\\
 & \le & 2\sum_{k>N\left(\epsilon\right)}k^{-\left(1+\epsilon\right)}\sum_{l\in\N}l^{-\left(1+\epsilon\right)}\end{eqnarray*}
 Then we have by integral comparison test for $M\in\N$ \begin{equation}
\sum_{k>M}k^{-\left(1+\epsilon\right)}\le\int_{M}^{\infty}x^{-\left(1+\epsilon\right)}dx=\frac{1}{\epsilon}M^{-\epsilon},\label{eq:keylemma integral criteria}\end{equation}
 Hence, for $0<\epsilon<1$, \[
\sum_{k>N\left(\epsilon\right)}k^{-\left(1+\epsilon\right)}\sum_{l\in\N}l^{-\left(1+\epsilon\right)}\le\frac{1}{\epsilon}N\left(\epsilon\right)^{-\epsilon}\left(\frac{1}{\epsilon}+1\right)\leq\frac{2}{\epsilon^{2}}N\left(\epsilon\right)^{-\epsilon}.\]
 With $N\left(\epsilon\right)=\left(\epsilon/3\right)^{-2/\epsilon}$
we get (A).

To verify (B) we use again (\ref{eq:keylemma integral criteria})
for $0<\epsilon<1$ to obtain\\
 \\
 ${\displaystyle \sums_{k\le N\left(\epsilon\right)\textrm{ and }l\le N\left(\epsilon\right)}\left(kl\right)^{-\left(1+\epsilon\right)}\left(1+\frac{1}{k\left(l+1\right)}\right)^{2\beta\left(\epsilon\right)}}$

\begin{eqnarray*}
 & \leq & \left(1+\frac{1}{N\left(\epsilon\right)\left(N\left(\epsilon\right)+1\right)}\right)^{2\beta\left(\epsilon\right)}\negmedspace\negmedspace\sum_{k\le N\left(\epsilon\right)}\sum_{l\le N\left(\epsilon\right)}\left(kl\right)^{-\left(1+\epsilon\right)}\\
 & \leq & \left(1+\frac{1}{2N\left(\epsilon\right)^{2}}\right)^{2\beta\left(\epsilon\right)}\frac{1}{\epsilon^{2}}.\end{eqnarray*}
 We are left to show that \begin{equation}
\left(1+\frac{1}{2N\left(\epsilon\right)^{2}}\right)^{2\beta\left(\epsilon\right)}\frac{1}{\epsilon^{2}}\le\frac{1}{2},\label{eq:keylem1}\end{equation}
 for $0<\epsilon<1/2$. Combining $N\left(\epsilon\right)=\left(\epsilon/3\right)^{-2/\epsilon}$
and $\beta\left(\epsilon\right)=(3/\log2)\cdot\log\left(\epsilon\right)\cdot\left(\epsilon/3\right)^{-4/\epsilon}$
gives $\beta\left(\epsilon\right)/N\left(\epsilon\right)^{2}=(3/\log2)\cdot\log\left(\epsilon\right)$.
Using this and the fact that $\log\left(1+x\right)\ge x\cdot\log(2),$
$x\in\left[0,1\right]$, we get \begin{eqnarray*}
\log\left(\left(1+\frac{1}{2N\left(\epsilon\right)^{2}}\right)^{2\beta\left(\epsilon\right)}\frac{1}{\epsilon^{2}}\right) & \le & 2\beta\left(\epsilon\right)\log\left(1+\frac{1}{2N\left(\epsilon\right)^{2}}\right)-2\log\left(\epsilon\right),\\
 & \leq & \log\left(2\right)\frac{\beta\left(\epsilon\right)}{N\left(\epsilon\right)^{2}}-2\log\left(\epsilon\right)\\
 & = & \log\left(\epsilon\right)<\log\left(1/2\right)\end{eqnarray*}
 for $0<\epsilon<1/2$. This proves (\ref{eq:keylem1}) and finishes
the proof of the lemma. 
\end{proof}

\section{Multifractal Analysis \label{sec:Multifractal-Analysis}}

In this section we prove our main theorems. In the first subsection
we prove the upper bound and in the second the lower bound for $f\left(\alpha\right)$.
For the upper bound we use a covering argument involving the $n$-th
partition function \[
Z_{n}\left(t,\beta\right):=\sum_{\omega\in\mathbb{N}^{n}}\exp\sup_{\tau\in[\omega]}\left(S_{n}t\varphi+\beta\psi\right)(\tau)\]
 which is also used to define the topological pressure $P\left(t,\beta\right)$.
To prove the lower bound we use the Thermodynamic Formalism to find
a measure $\mu$ such that on the one hand $\int\psi\, d\mu/\int\phi\, d\mu=\alpha$
and on the other hand $\mu$ maximises the quotient of the metrical
entropy $h_{\mu}$ and the Lyapunov exponent $\int\phi\, d\mu$. It
will turn out that this measure is in fact the equilibrium measure
for the potential $t\left(\beta\right)\phi+\beta\psi$. 

In the last subsection we prove Proposition \ref{pro:ramharter prop}
and analyse the boundary points of the spectrum. This part makes extensive
use of some number theoretical estimates depending heavily on the
recursive nature of the Diophantine approximation.

\subsection{Upper bound}

For the upper bound we apply a covering argument to the set $\mathcal{F}_{\alpha}^{*}$.
\begin{prop}
\label{pro:upperbound}For $\alpha\in\R$ we have \[
\dim_{H}\left(\mathcal{F}_{\alpha}\right)\leq\dim_{H}\left(\mathcal{F}_{\alpha}^{*}\right)\le\max\left\{ \inf_{\beta\in\R}\left\{ t(\beta)+\beta\alpha\right\} ,0\right\} .\]
 If there exists $\beta\in\R$, such that $t\left(\beta\right)+\beta\alpha<0$
then we have $\mathcal{F}_{\alpha}^{*}=\emptyset$. \end{prop}
\begin{proof}
The first inequality follows from $\mathcal{F}_{\alpha}\subset\mathcal{F}_{\alpha}^{*}$.
For the second we make the following assumption. For all $\beta\in\R$
and $\epsilon>0$ we have $\mathcal{H}^{t(\beta)+\beta\alpha+\epsilon}\left(\mathcal{F}_{\alpha}^{*}\right)<\infty$,
where $\mathcal{H}^{s}$ denotes the $s$-dimensional Hausdorff measure
(see \cite{falconerfractalgeometryMR2118797} for this and related
notions from fractal geometry). If then $t(\beta)+\beta\alpha\ge0$
we can conclude, that $\dim_{H}\left(\mathcal{F}_{\alpha}^{*}\right)\le t(\beta)+\beta\alpha$.
If on the other hand there exists $\beta\in\R$ such that $t\left(\beta\right)+\beta\alpha<0,$
then we would have $\mathcal{H}^{s}\left(\mathcal{F}_{\alpha}^{*}\right)<\infty$
for some $s<0$. This clearly gives $\mathcal{F}_{\alpha}^{*}=\emptyset$
and consequently $\dim_{H}\left(\mathcal{F}_{\alpha}^{*}\right)=0$.

Now we are left to prove the assumption. We will only consider the
case $\alpha\geq\alpha_{0}$ (the case $\alpha<\alpha_{0}$ can be
treated in a completely analogous way). Then with out loss of generality
we may assume that $\beta\leq0$ (otherwise $t(\beta)+\beta\alpha\geq1$)
. For $r,\delta>0$ fixed we are going to construct a $\delta$-covering
of $\Pi\left(\mathcal{F}_{\alpha}^{*}\right)$. Since the Gauss system
is uniformly contractive, for each $\omega\in\pi\left(\mathcal{F}_{\alpha}^{*}\right)$
there exists $n(\omega,\delta,r)$ such that \begin{equation}
\frac{S_{n(\omega,\delta,r)}\psi\left(\omega\right)}{S_{n(\omega,\delta,r)}\varphi\left(\omega\right)}\geq\alpha-r\label{eq:snpsi durch snphi in r umgebung von alpha}\end{equation}
 and \begin{equation}
\diam\left(\pi^{-1}[\omega_{|n(\omega,\delta,r)}]\right)<\delta.\label{eq:cylinder kleiner delta}\end{equation}
 We surely have $\pi\left(\mathcal{F}_{\alpha}^{*}\right)\subset\bigcup_{\omega\in\Pi\left(\mathcal{F}_{\alpha}\right)}\omega_{|n(\omega,\delta,r)}$.
Removing duplicates from the cover, we obtain an at most countable
$\delta$-cover $\omega_{|n(\omega^{(i)},\delta,r)}^{(i)}$ with $i\in\N$,
because there are only countably many finite words over a countable
alphabet.

We will now prove $\mathcal{H}^{t(\beta)+\beta\alpha+\epsilon}\left(\mathcal{F}_{\alpha}^{*}\right)<\infty$
for fixed $\epsilon>0$. Using the cover constructed above we have
by the bounded distortion property (\ref{eq:boundeddistortion}) that
there exists a constant $C>0$ such that \begin{eqnarray*}
\mathcal{H}_{\delta}^{t(\beta)+\beta\alpha+\epsilon}\left(\mathcal{F}_{\alpha}^{*}\right) & \le & \sum_{i\in\N}\diam\left(U_{i}(\delta,r)\right)^{t(\beta)+\beta\alpha+\epsilon}\\
 & = & \sum_{i\in\N}\diam\left(\pi^{-1}[\omega_{|n(\omega^{(i)},\delta,r)}^{(i)}]\right)^{t(\beta)+\beta\alpha+\epsilon}\\
 & \le & C\sum_{i\in\N}\exp\left[S_{n(\omega^{(i)},\delta,r)}\varphi(\omega^{(i)})\left(t(\beta)+\beta\alpha+\epsilon\right)\right].\end{eqnarray*}
 Now choose $r>0$ so small, such that for all $i\in\N$ we have \[
\beta\alpha+\epsilon/2>\beta\frac{S_{n(\omega^{(i)},\delta,r)}\psi(\omega^{(i)})}{S_{n(\omega^{(i)},\delta,r)}\varphi(\omega^{(i)})}.\]
Since $S_{n}\varphi<0$ we have\\
 \\
 $\mathcal{H}_{\delta}^{t(\beta)+\beta\alpha+\epsilon}\left(\mathcal{F}_{\alpha}^{*}\right)$\begin{eqnarray*}
 & \leq & C\sum_{i\in\N}\exp\left[S_{n(\omega^{(i)},\delta,r)}\varphi(\omega^{(i)})\left(t(\beta)+\beta\frac{S_{n(\omega^{(i)},\delta,r)}\psi(\omega^{(i)})}{S_{n(\omega^{(i)},\delta,r)}\varphi(\omega^{(i)})}+\frac{\epsilon}{2}\right)\right]\\
 & = & C\sum_{i\in\N}\exp\left[S_{n(\omega^{(i)},\delta,r)}\varphi(\omega^{(i)})\left(t(\beta)+\frac{\epsilon}{2}\right)+\beta S_{n(\omega^{(i)},\delta,r)}\psi(\omega^{(i)})\right]\\
 & = & C\sum_{i\in\N}\exp\left[S_{n(\omega^{(i)},\delta,r)}\left(\left(t(\beta)+\frac{\epsilon}{2}\right)\varphi+\beta\psi\right)(\omega^{(i)})\right]\\
 & \leq & C\sum_{n\in\N}\sum_{\omega\in\Sigma^{n}}\exp\sup_{\tau\in[\omega]}\left[S_{n}\left(\left(t(\beta)+\frac{\epsilon}{2}\right)\varphi+\beta\psi\right)\right].\end{eqnarray*}
 Since we have $P\left(t(\beta),\beta\right)=0$ by definition of
$t\left(\beta\right)$ and the fact that the pressure $P$ is strictly
decreasing with respect to the first component (see (\ref{eq:pressure bounded away from zero})
in the proof of Proposition \ref{pro:exTf(beta)}), we conclude that
$P\left(t(\beta)+\epsilon/2,\beta\right)=\eta<0$. This implies \[
\sum_{\omega\in\Sigma^{n}}\exp\sup_{\tau\in[\omega]}\left(S_{n}\left(\left(t(\beta)+\frac{\epsilon}{2}\right)\varphi+\beta S_{n}\psi\right)\right)\ll\exp\left(n\frac{\eta}{2}\right).\]
Hence, there exists another positive constant $C'$ such that for
all $\delta>0$ we have $\mathcal{H}_{\delta}^{t(\beta)+\beta\alpha+\epsilon}\left(\mathcal{F}_{\alpha}^{*}\right)\le C'\sum_{n\in\N}\exp\left(n\eta/2\right)<\infty$.
This implies $\mathcal{H}^{t(\beta)+\beta\alpha+\epsilon}\left(\mathcal{F}_{\alpha}^{*}\right)<\infty$
showing that $\dim_{H}\left(\mathcal{F}_{\alpha}^{*}\right)\le t(\beta)+\beta\alpha+\epsilon$
. The claim follows by letting $\epsilon$ tend to zero. 
\end{proof}

\subsection{Lower bound}

For the lower bound we use the Volume Lemma (\cite[Theorem 4.4.2]{urbanskimauldin-gdmsMR2003772}),
which in our situation can be stated as follows. Let $\mu$ be a $\sigma$-invariant
probability on $\N^{\N}$ such that either $\sum_{k\in\N}\mu\left(\left[k\right]\right)\log\left(\mu\left(\left[k\right]\right)\right)<\infty$
or $\int\phi\, d\mu<\infty$. Then \begin{equation}
\HD\left(\mu\right)=\frac{h_{\mu}}{\int\varphi\, d\mu},\label{eq:8}\end{equation}
where $\HD\left(\mu\right):=\inf\left\{ \dim_{H}\left(Y\right):Y\subset\mathbb{I},\,\mbox{measurable, }\mu(Y)=1\right\} $.
In the following $\Im\left(g\right)$ will denote the image of the
function $g$. 
\begin{prop}
\label{pro:lowerbound}For $\alpha\in-\Im(t')$ we have \[
\dim_{H}\left(\mathcal{F}_{\alpha}^{*}\right)\geq\dim_{H}\left(\mathcal{F}_{\alpha}\right)\ge\inf_{\beta\in\R}\left\{ t(\beta)+\beta\alpha\right\} >0.\]
\end{prop}
\begin{proof}
Again, as for the upper bound, the first inequality is immediate.
For $\alpha\in-\Im(t')$ let $\beta=\left(t'\right)^{-1}\left(-\alpha\right)$.
Since $\varphi\in\mathcal{L}^{1}\left(\mu_{t\left(\beta\right)\varphi+\beta\psi}\right)$
we have by the Volume Lemma, Proposition \ref{pro:existence of invariant gibbs},
the fact that $P\left(t\left(\beta\right),\beta\right)=0$, and (\ref{eq:7})
that \begin{eqnarray*}
\HD\left(\mu_{\beta}\right) & = & \frac{h_{\mu_{\beta}}}{\int\varphi\, d\mu_{\beta}}=\frac{\int t\left(\beta\right)\varphi+\beta\psi\, d\mu_{\beta}}{\int\varphi\, d\mu_{\beta}}\\
 & = & t\left(\beta\right)-\beta t'\left(\beta\right)\\
 & = & t\left(\left(t'\right)^{-1}\left(-\alpha\right)\right)+\left(t'\right)^{-1}\left(-\alpha\right)\alpha\\
 & = & -\hat{t}\left(-\alpha\right),\end{eqnarray*}
where the last equality holds by \cite[Theorem 26.4]{rockafellar-convexanalysisMR0274683}.
By (\ref{eq:8}) we have $-\hat{t}\left(-\alpha\right)\ge0$. Furthermore,
since $t$ is strictly convex (Proposition \ref{pro:exTf(beta)})
we conclude with \cite[Corollary 26.4.1]{rockafellar-convexanalysisMR0274683}
that also the Legendre conjugate $\hat{t}$ is strictly convex on
$\Im(t')$. Hence, for $\alpha\in-\Im(t')$ we have \begin{equation}
\HD\left(\mu_{\beta}\right)=\inf_{c\in\R}\left\{ t(c)+c\alpha\right\} >0.\label{eq:9}\end{equation}
 Since $\mu_{\beta}$ is ergodic (Proposition \ref{pro:existence of invariant gibbs})
we have by the Ergodic Theorem, the choice of $\beta$ and (\ref{eq:7})
that \[
\lim_{n\rightarrow\infty}\frac{S_{n}\psi\left(\omega\right)}{S_{n}\varphi\left(\omega\right)}=\frac{\int\psi\, d\mu_{\beta}}{\int\varphi\, d\mu_{\beta}}=\alpha\quad\quad\mbox{for }\mu_{\beta}\mbox{-a.e.}\;\omega.\]
 This gives $\mu_{\beta}\left(\mathcal{F}_{\alpha}\right)=1$, which
together with (\ref{eq:9}) and the definition of $\HD\left(\mu_{\beta}\right)$
finishes the proof. 
\end{proof}
Now we can prove the main theorem neglecting the boundary points.

\emph{Proof of first part of Theorem \ref{thm:maintheorem}}.  Clearly,
$\mathcal{F_{\alpha}}\subset\mathcal{F}_{\alpha}^{*}$. Combining
Proposition \ref{pro:upperbound} and Proposition \ref{pro:lowerbound}
gives $f\left(\alpha\right)=\inf_{\beta\in\R}\left\{ t(\beta)+\beta\alpha\right\} $
for $\alpha\in-\Im\left(t'\right)$. Since also by Proposition \ref{pro:lowerbound}
$f\left(\alpha\right)>0$ for $\alpha\in-\Im\left(t'\right)$ we conclude
that $-\Im\left(t'\right)$ (which is an open set) is contained in
$\left(0,1\right)$. Furthermore, for $\alpha\notin-\overline{\Im\left(t'\right)}$,
we have $\inf_{\beta\in\R}\left\{ t(\beta)+\beta\alpha\right\} =-\infty$
(\cite[Corollary 26.4.1]{rockafellar-convexanalysisMR0274683}), hence
by Proposition \ref{pro:upperbound} we have $\mathcal{F}_{\alpha}=\emptyset$.
Since $\mathcal{F}_{0}$ and $\mathcal{F}_{1}$ are not empty, we
have $-\Im\left(t'\right)=\left(0,1\right)$. Notice that $f\left(\alpha\right)=-\widehat{t}\left(-\alpha\right)$
for $\alpha\in\left(0,1\right)$ and by \cite[Theorem 26.5]{rockafellar-convexanalysisMR0274683}
\begin{equation}
f'\left(\alpha\right)=\left(\widehat{t}\right)'\left(-\alpha\right)=\left(t'\right)^{-1}\left(-\alpha\right).\label{eq:10}\end{equation}
 Since $t'$ is strictly increasing we conclude, that $f$ is strictly
concave and by the inverse function theorem that $f$ is real-analytic.
\qed

\subsection{Boundary points}

In the last section we finish the proof of Theorem \ref{thm:maintheorem}
and give a proof of Proposition \ref{pro:ramharter prop} and Theorem
\ref{thm:Asymp}.

\emph{Proof of the remaining parts of Theorem \ref{thm:maintheorem}.}
We have to show 
\begin{itemize}
\item [(a)]\emph{$\lim_{\alpha\searrow0}f'\left(\alpha\right)=\infty$
and $\lim_{\alpha\nearrow1}f'\left(\alpha\right)=-\infty$}, 
\item [(b)] $\lim_{\alpha\searrow0}\dim_{H}\left(\mathcal{F}_{\alpha}\right)=\dim_{H}\left(\mathcal{F}_{0}\right)=0$, 
\item [(c)] $\lim_{\alpha\nearrow1}\dim_{H}\left(\mathcal{F}_{\alpha}\right)=\dim_{H}\left(\mathcal{F}_{1}\right)=1/2.$ 
\end{itemize}
The assertion in (a) follows directly from equation (\ref{eq:10}).
To prove (b) notice that by (\ref{eq:pressure bounded away from zero})
and the definition of $t$ we have for $\beta\in\R$ \[
0=P\left(t\left(\beta\right),\beta\right)\le-2t\left(\beta\right)\log\left(\gamma\right)+P\left(0,\beta\right)\]
 which implies $t\left(\beta\right)\le\frac{P\left(0,\beta\right)}{2\log\left(\gamma\right)}$.
Since $P\left(0,\beta\right)=\lg\left(\zeta\left(2\beta\right)\right)$,
which tends to zero for $\beta\rightarrow\infty$, we conclude that
$\lim_{\beta\rightarrow\infty}$$t\left(\beta\right)=0$. By the upper
bound in Proposition \ref{pro:upperbound}, we have that $\dim_{H}\left(\mathcal{F}_{\alpha}\right)$
is dominated by $\inf_{\beta\in\R}\left\{ t\left(\beta\right)+\beta\alpha\right\} $,
which becomes arbitrarily small for $\alpha\searrow0$ and which is
equal to zero for $\alpha=0$.

To prove the lower bounds in part (c) of the proposition we first
notice that for $\alpha=-t'\left(\beta\right)$ such that $1>\alpha>-t'\left(0\right)$
we have $\beta<0$. By the lower bound in Proposition \ref{pro:lowerbound}
we have on the one hand $\dim_{H}\left(\mathcal{F}_{\alpha}\right)\ge t\left(\beta\right)+\beta\alpha$.
By Lemma \ref{lem:pressureFiniteIFF} we have $P\left(t,\beta\right)<\infty$,
if and only if $t+\beta>1/2$. Since $P\left(t\left(\beta\right),\beta\right)=0<\infty$
we can conclude that on the other hand we have $t\left(\beta\right)>1/2-\beta$.
Combining these two observations we have $\dim_{H}\left(\mathcal{F}_{\alpha}\right)\ge t\left(\beta\right)+\beta\alpha>1/2-\beta\left(1-\alpha\right)>1/2$
for $\alpha\in\left(-t'\left(0\right),1\right)$. Since $\mathcal{G}\subset\mathcal{F}_{1}$
it follows directly from (\ref{eq:increasing sequ dim 1 2}) that
$\dim_{H}\left(\mathcal{F}_{1}\right)\ge1/2$.

To finally prove the upper bounds in (c) fix $\epsilon>0$. Lemma
\ref{lem:tAsymptoticFor-infty} guarantees that there exists $\beta_{0}\in\R$
such that for all $\beta\le\beta_{0}$ we have $t\left(\beta\right)<1/2-\beta+\epsilon$.
Using Proposition \ref{pro:upperbound} we have for $\alpha\in\left(0,1\right)$
\begin{eqnarray*}
f\left(\alpha\right)=\dim_{H}\left(\mathcal{F}_{\alpha}\right) & \le & \inf_{c\in\R}\left\{ t\left(c\right)+c\alpha\right\} \le\inf_{c\le\beta_{0}}\left\{ 1/2+\epsilon-c\left(1-\alpha\right)\right\} \\
 & \leq & 1/2+\epsilon-\beta_{0}\left(1-\alpha\right)\to1/2+\epsilon\qquad\mbox{for }\alpha\nearrow1.\end{eqnarray*}
 Since $\epsilon>0$ was arbitrary we have both $\limsup_{\alpha\nearrow1}\dim_{H}\left(\mathcal{F}_{\alpha}\right)\le1/2$
and $\dim_{H}\left(\mathcal{F}_{1}\right)\le1/2$. In particular,
since $\hat{t}$ is continuous on $\left[0,1\right]$ it follows that
$f$ and $a\mapsto-\hat{t}\left(-a\right)$ agree on $\left[0,1\right]$.
\qed

\emph{Proof of Proposition \ref{pro:ramharter prop}.} We are going
to apply our multifractal formalism to the Gauss system restricted
to the state space $\mathcal{I}_{q}$, $q\in\N$. In particular, we
introduce the restricted pressure\[
P_{q}\left(t,\beta\right):=\lim_{n\rightarrow\infty}\frac{1}{n}\lg\sum_{\omega\in\left\{ q,q+1,\ldots\right\} ^{n}}q_{n}\left(\omega\right)^{-2t}\prod_{i=1}^{n}\omega_{i}^{-2\beta},\quad t,\beta\in\R.\]
 Arguing as in the proof of Proposition \ref{pro:exTf(beta)}, we
find a real-analytic function $t_{q}\colon\R\to\R$ such that $P_{q}\left(t_{q}\left(\beta\right),\beta\right)=0$
for all $\beta\in\R$. By Bowen's Formula (cf. \cite[Theorem 4.2.13]{urbanskimauldin-gdmsMR2003772})
we have that \[
\dim_{H}\left(\mathcal{I}_{q}\right)=\inf\left\{ t\in\R:P\left(t\phi\right)<0\right\} =t_{q}\left(0\right).\]
Using $\prod_{k=1}^{n}a_{k}\left(x\right)\leq q_{n}\left(x\right)\leq\prod_{k=1}^{n}\left(a_{k}\left(x\right)+1\right)$
and $t_{q}\left(0\right)\geq0$ we find\[
\sum_{k\ge q+1}k^{-2\left(t_{q}\left(0\right)-1/2\right)-1}\le e^{P_{q}\left(t_{q}\left(0\right),0\right)}\le\sum_{k\ge q}k^{-2\left(t_{q}\left(0\right)-1/2\right)-1}.\]
 Integral comparison test gives \[
\frac{1}{2\left(t_{q}\left(0\right)-1/2\right)}\left(q+1\right)^{-2\left(t_{q}\left(0\right)-1/2\right)}\le1\le\frac{1}{2\left(t_{q}\left(0\right)-1/2\right)}\left(q-1\right)^{-2\left(t_{q}\left(0\right)-1/2\right)},\]
 which is equivalent to \[
q-1\le\left(2\cdot\left(t_{q}\left(0\right)-1/2\right)\right)^{-1/\left(2\cdot\left(t_{q}\left(0\right)-1/2\right)\right)}\le q+1.\]
 This proves $t_{q}\left(0\right)-1/2\thicksim1/2\cdot\log\left(\log\left(q\right)\right)/\log\left(q\right)$.
\qed

\emph{Proof of Theorem \ref{thm:Asymp}.} Using Proposition \ref{pro:upperbound}
with $\beta\colon\epsilon\mapsto\frac{3}{\log(2)}\log\left(\epsilon\right)\left(\frac{\epsilon}{3}\right)^{-\frac{4}{\epsilon}}$
from Lemma \ref{lem:tAsymptoticFor-infty} we have for $\epsilon<1/2$
and $\delta\in\left(0,1\right)$\begin{eqnarray*}
f\left(1-\delta\right)=\dim_{H}\left(\mathcal{F}_{\alpha}^{*}\right) & \le & \inf_{c\in\R}\left\{ t\left(c\right)+c\left(1-\delta\right)\right\} \le1/2+\frac{\epsilon}{2}-\beta\left(\epsilon\right)\delta.\end{eqnarray*}
 Now with $\epsilon\left(\delta\right):=4\log\left(\log\left(1/\delta\right)\right)/\log\left(1/\delta\right)$,
we have for $\delta\to0$\\
\\
${\displaystyle \frac{-\beta\left(\epsilon\left(\delta\right)\right)}{\epsilon\left(\delta\right)}\delta}$\[
=-\frac{3\delta\log(1/\delta)\left(2\log(2)+\log\left(\frac{\log(\log(1/\delta)}{\log\left(1/\delta\right)}\right)\right)\left(\frac{4\log(\log(1/\delta))}{3\log\left(1/\delta\right)}\right)^{\frac{-\log(1/\delta)}{\log(\log(1/\delta))}}}{4\log\left(2\right)\log(\log(1/\delta))}\to0.\]
 This proves $f\left(1-\delta\right)\leq1/2+d\cdot\log\log\left(1/\delta\right)/\log\left(1/\delta\right)$
for any $d>2$ and for $\delta>0$ sufficiently small.

For the proof of the lower bound we make use of Fact \ref{fact6}
and Proposition \ref{pro:ramharter prop}. First we show that for
sufficiently small $\delta>0$ we have \begin{equation}
\mathcal{I}_{q\left(\delta\right)}\subset\mathcal{F}_{1-\delta}^{*}\;\textrm{ where }\; q\left(\delta\right):=\left(\frac{\delta}{3}\cdot\log\left(1/\delta\right)\right)^{-\frac{1}{2}}.\label{eq:q vs delta}\end{equation}
 In fact, (\ref{eq:q vs delta}) follows from Fact \ref{fact6} since
for $\delta\to0$ we have \begin{eqnarray*}
1-\alpha_{q\left(\delta\right)} & = & \left(\left(q\left(\delta\right)\right)^{2}\log\left(q\left(\delta\right)\right)\right)^{-1}\\
 & = & \left(\frac{\delta}{3}\cdot\log\left(1/\delta\right)\right)\cdot\left(\log\left(\left(\frac{\delta}{3}\cdot\log\left(1/\delta\right)\right)^{-\frac{1}{2}}\right)\right)^{-1}\\
 & = & \frac{\delta}{3}\cdot\log\left(1/\delta\right)\cdot\left(-2\right)\cdot\left(\log\left(\frac{\delta}{3}\cdot\log\left(1/\delta\right)\right)\right)^{-1}\sim\frac{2}{3}\delta.\end{eqnarray*}

Now by Proposition \ref{pro:ramharter prop} and (\ref{eq:q vs delta})
we have for $c<1/2$ and sufficiently small $\delta>0$ \begin{eqnarray*}
\dim_{H}\left(\mathcal{F}_{1-\delta}^{*}\right) & \ge & \frac{1}{2}+c\left(\frac{\log\log q\left(\delta\right)}{\log q\left(\delta\right)}\right)\\
 & \ge & \frac{1}{2}+c\cdot\frac{\log\log\left(\left(\delta/3\cdot\log\left(1/\delta\right)\right)^{-\frac{1}{2}}\right)}{\log\left(\delta/3\cdot\log\left(1/\delta\right)\right)^{-\frac{1}{2}}}\\
 & = & \frac{1}{2}+2c\cdot\frac{\log\left(1/2\cdot\log\left(3/\delta\right)-1/2\cdot\log\log\left(1/\delta\right)\right)}{\log\left(3/\delta\right)-\log\log\left(1/\delta\right)}.\end{eqnarray*}
 Since \[
\frac{\log\left(1/2\cdot\log\left(3/\delta\right)-1/2\cdot\log\log\left(1/\delta\right)\right)}{\log\left(3/\delta\right)-\log\log\left(1/\delta\right)}\sim\frac{\log\left(\log\left(1/\delta\right)\right)}{\log\left(1/\delta\right)}\]
 for $\delta\to0$ the result follows. \qed
\begin{acknowledgement*}
We would like to thank the referee for useful comments that helped
to improve the presentation of this paper significantly. 
\end{acknowledgement*}

\end{document}